\documentclass{article}
\usepackage{graphicx} % Required for inserting images
\usepackage{amssymb,amsfonts,amscd}
\usepackage[all,arc]{xy}
\usepackage{enumerate}
\usepackage{graphicx}
\usepackage{mathrsfs}
\usepackage{amsmath}
\usepackage{amsthm}
\usepackage{amsfonts}
\usepackage[utf8]{inputenc}
\usepackage[margin=1.2in]{geometry}
\usepackage[font=small,format=plain,labelfont=bf,up,textfont=it,up]{caption}
\usepackage{tikz-cd}
\newtheorem{thm}{Theorem}[section]
\newtheorem{cor}[thm]{Corollary}
\newtheorem{prop}[thm]{Proposition}
\newenvironment{pf}{{\noindent\textbf{Proof}}\quad}{\hfill $\square$}
\newtheorem{lem}[thm]{Lemma}

\newtheorem{defn}[thm]{Definition}

\theoremstyle{definition}

\newtheorem*{remark}{Remark}
\newcommand{\ra}{\rightarrow}
\newcommand{\Q}{\mathbb{Q}}
\newcommand{\A}{\mathbb{A}}
\newcommand{\B}{\mathbb{B}}
\newcommand{\bdr}{\mathbb{B}_\mathrm{dR}}
\newcommand{\dr}{\mathrm{dR}}
\newcommand{\Z}{\mathbb{Z}}
\newcommand{\N}{\mathbb{N}}
\renewcommand{\P}{\mathbb{P}}
\newcommand{\X}{\mathcal{X}}
\newcommand{\D}{\mathcal{D}}
\renewcommand{\O}{\mathcal{O}}
\newcommand{\la}{\mathrm{la}}
\newcommand{\sm}{\mathrm{sm}}
\newcommand{\HT}{\mathrm{HT}}
\newcommand{\GL}{\mathrm{GL}}
\newcommand{\gal}{\mathrm{Gal}}
\newcommand{\Gal}{\mathrm{Gal}}
\newcommand{\abcd}{\left(\begin{smallmatrix}
        a&b\\c&d
\end{smallmatrix}\right)}
\newcommand{\ma}[1]{\left(\begin{smallmatrix}
        #1
\end{smallmatrix}\right)}
\newcommand{\gl}{\mathfrak{gl}}
\renewcommand{\b}{\mathfrak{b}}
\newcommand{\n}{\mathfrak{n}}
\newcommand{\h}{\mathfrak{h}}
\newcommand{\pa}{\partial}
\newcommand{\fl}{\mathscr{F}\ell}

\title{The Fontaine operator at cusps of 
modular curves at infinite level}
\author{Tian Qiu}
\date{}

\begin{document}

\maketitle
\begin{abstract}
    % In \cite{pan2022}, Lue Pan constructed a geometric  intertwining operator on modular curves with infinite level at $p$
    % %differential operators on modular curves with infinite level at $p$ in both ``holomorphic" and ``anti-holomorphic" directions. He proved that the composite of these two differential operators is 
    % and proved that it coincides with the Fontaine operator arising from $p$-adic Hodge theory.
    % %on the locally analytic vectors of the completed cohomology of modular curves and 
    % He used it to relate the de Rhamness to the classicality of two-dimensional Galois representations of $\Gal_\Q$, and to provide a detailed description of the locally analytic vectors in the completed cohomology associated to such  representations. In this paper, we use the theory of cusps for modular curves at infinite level developed in \cite{heuer2022cusps} to explicitly calculate these operators via $q$-expansions at infinite level. In particular, we prove that the intertwining operator and the Fontaine operator still coincide up to a constant on such expansions, and we determine this constant. By combining this result with $q$-expansion principles at infinite level, we reprove Pan's result that these operators are equal on the locally analytic vectors of the completed cohomology of modular curves.

    We explicitly calculate Pan's geometric intertwining operator and the Fontaine operator on modular curves at infinite level via $q$-expansions, using Heuer's theory of cusps at infinite level. We prove that these two operators coincide on such expansions up to an explicit constant. As an application, we combine this result with $q$-expansion principles to provide a new proof of Pan's theorem that these operators are equal on the locally analytic vectors of completed cohomology of modular curves.
\end{abstract}
\section{Introduction}
%\subsection*{Differential Operators  for Modular Curves at Infinite Level}
    Let $p$ be a prime and $C=\widehat{\overline{\Q}}_p$. Let $K^p\subset \mathrm{GL}_2(\A_f^p)$ be a fixed tame level. For an open compact subgroup $K_p\subset \mathrm{GL}_2(\Q_p)$, let $X_{K^pK_p}$ be the modular curve over $\Q$ of level $K^pK_p$. In his pioneering work \cite{Eme06}, Emerton introduced  completed cohomology 
    \[\tilde{H}^1(K^p,\Q_p):=(\varprojlim_n\varinjlim_{K_p\subset \mathrm{GL}_2(\Q_p)}H^1_{\text{\'et}}(X_{K^pK_p,\overline{\Q}},\Z/p^n\Z))[\frac{1}{p}].\]
    This is a $p$-adic Banach space equipped with natural continuous actions of $\GL_2(\Q_p)\times \gal_{\Q}$. Let $E$ be a finite extension of $\Q_p$ and let  
    \[\rho:\Gal_\Q\ra \GL_2(E)\]
    be a two-dimensional continuous absolutely irreducible representation such that $\rho|_{\gal_{\Q_p}}$ is de Rham of Hodge-Tate weights $0,k$ for some integer $k>0$. 
    % We assume that the $\rho$-isotypic component $\tilde{H}^1(K^p,E)[\rho]$ of  $\tilde{H}^1(K^p,E):=\tilde{H}^i(K^p,\Q_p)\otimes_{\Q_p}E$ is non-zero and $\rho|_{\gal_{\Q_p}}$ is de Rham of Hodge-Tate weights $0,k$ for some integer $k>0$. 
    By the work of Emerton \cite{emerton2011local}, the $\rho$-isotypic component $\tilde{H}^1(K^p,E)[\rho]$ of $\tilde{H}^1(K^p,E):=\tilde{H}^1(K^p,\Q_p)\otimes_{\Q_p}E$ realizes the $p$-adic local Langlands correspondence associated to $\rho|_{\gal_{\Q_p}}$. Let $\mathbb{T}(K^p)$ be the Hecke algebra of tame level $K^p$ (see Section \ref{app} for details). It acts faithfully on $\tilde{H}^1(K^p,E)$ and there exists a homomorphism $\lambda:\mathbb{T}(K^p)\ra E$ such that the $\rho$-isotypic component of $\tilde{H}^1(K^p,E)$  coincides with the $\lambda$-isotypic component, i.e.
\[\tilde{H}^1(K^p,E)[\lambda]=\tilde{H}^1(K^p,E)[\rho]. \]
    
    Let $\tilde{H}^1(K^p,E)[\lambda]^\la$ be the $\GL_2(\Q_p)$-locally analytic vectors in $\tilde{H}^1(K^p,E)[\lambda]$. 
    In \cite{panI},\cite{pan2022}, Lue Pan used geometric methods to study the structure of $\tilde{H}^1(K^p,E\otimes_{\Q_p}C)[\lambda]^\la$. More explicitly, let
    $\X_{K^pK_p}$ be the adic space associated to $X_{K^pK_p,C}$. 
     Let $\mathcal{X}_{K^p}\sim \varprojlim_{K_p\subset \GL_2(\Q_p)}\mathcal{X}_{K^pK_p}$ denote the modular curve at infinite level introduced by Scholze in \cite{Sch15}. There is a $\GL_2(\Q_p)$-equivariant Hodge--Tate period map
    $$\pi_{\HT}:\mathcal{X}_{K^p}\ra \mathscr{F}\ell\cong \mathbb{P}^1$$
    where $\fl$ denotes the flag variety of $\GL_2$. Let 
    $\mathcal{O}_{K^p}:=\pi_{\HT,*}\O_{\mathcal{X}_{K^p}}$ and let $\O_{K^p}^\la\subset \O_{K^p}$ be the subsheaf of $\GL_2(\Q_p)$-locally analytic sections. In \cite{pan2022}, he proved that there is a horizontal action 
    $\theta_\mathfrak{h}$ of $\mathfrak{h}=\{(\begin{smallmatrix}
        * &0\\0&*
    \end{smallmatrix})\}\subset \mathfrak{gl}_2(\Q_p)$ on $\O^\la_{K^p}$ which relates both the infinitesimal action and the Hodge--Tate weights. Moreover, there are natural isomorphisms
    \[\tilde{H}^1(K^p,E)[\lambda]^\la\otimes_{\Q_p} C\cong H^1(\fl,\O_{K^p}^{\la,(1-k,0)})[\lambda]\oplus H^1(\fl,\O_{K^p}^{\la,(1,-k)})[\lambda],\]
%\begin{equation*}
    %W_0\cong H^1(\fl,\O_{K^p}^{\la,(1-k,0)}),\quad  W_k\cong H^1(\fl,\O_{K^p}^{\la,(1,-k)})
%\end{equation*}
where $\O_{K^p}^{\la,(n_1,n_2)}$ denotes the $(n_1,n_2)$-isotypic part of $\O_{K^p}^{\la}$ with respect to $\theta_\mathfrak{h}$ and $H^1(\fl,\O_{K^p}^{\la,(1-k,0)})[\lambda]$ (resp. $H^1(\fl,\O_{K^p}^{\la,(1,-k)})[\lambda]$) is the Hodge-Tate weight $0$ (resp. $k$) part of $\tilde{H}^1(K^p,E)[\lambda]^\la\otimes_{\Q_p} C$.
On $\O_{K^p}^{\la,(1-k,0)}$, Pan defined the intertwining operator
$$I_{k-1}:\O_{K^p}^{\la,(1-k,0)}\ra \O_{K^p}^{\la,(1,-k)}(k).$$
Intuitively, this is the composition of $k$-th differential along the modular curve and $k$-th differential along the flag variety.

Using $p$-adic Hodge theory, he also constructed a map
\[N_k:  \O_{K^p}^{\la,(1-k,0)}\ra \O_{K^p}^{\la,(1,-k)}(k),\]
which is referred to as the Fontaine operator (see Section \ref{fon} for details). It is related to de Rhamness of the associated Galois  representation. In \cite{pan2022}, Pan proved that $N_k$ and $I_{k-1}$ coincide up to a unit and he utilized this result to study the structure of $\tilde{H}^1(K^p,E\otimes_{\Q_p}C)[\lambda]^\la$.

\iffalse
two differential operators: one is the differential along the modular curve
$$d^k:\O_{K^p}^{\la,(1-k,0)}\ra \O_{K^p}^{\la,(1-k,0)}\otimes_{\O_{K^p}^{\sm}}(\Omega^1_{K^p}(\mathcal{C})^\sm)^{\otimes k}.$$
%which comes from the differential map $\theta:\omega^{-k+1}\ra \omega^{k+1}$ at finite levels of modular curves. 
Another is the differential along the flag variety
$$\overline{d}^k:\O_{K^p}^{\la,(1-k,0)}\ra \O_{K^p}^{\la,(1-k,0)}\otimes_{\O_{\fl}}(\Omega^1_{\fl})^{\otimes k}.$$
%which is given by $(u^+)^k$ where $u^+=(\begin{smallmatrix}
   % 0&1\\0&0
%\end{smallmatrix})\in \mathfrak{gl}_2(\Q_p)$ on the locus away from $\infty\in \fl$.
By Kodaira-Spencer isomorphism, their composite defines a map
$$I_{k-1}:=d^k\circ \overline{d}^k=\overline{d}^k\circ d^k:\O_{K^p}^{\la,(1-k,0)}\ra \O_{K^p}^{\la,(1,-k)}(k).$$
and it induces a map between the cohomology
$I^1:W_0\ra W_k(k).$
In \cite{pan2022}, Pan proved the following result.
\begin{thm}
    $N=cI^1$ for some $c\in\Q_p^\times$.
\end{thm}
\fi

%\subsection*{Cusps for Modular Curves at Infinite Level}
In this paper, we consider an analogue of Pan's construction at the cusps at infinite level and provide an explicit formula. Assume that $\Gamma(N)\subset K^p$ for some $N\in \Z$ coprime to $p$. 
For each cusp $x$ of $\X:=\X_{K^p\GL_2(\Z_p)}$, it corresponds to a $K^p$-level structure on the Tate curve $T(q^e)$ over $\O_C((q))$ for some positive integer $e|N$. The analytic Tate curve parameter space induces a canonical open immersion $\D\hookrightarrow \X$ sending the origin to $x$, where $\D$ is the adic open unit disc. Let $\D_\infty$ be the perfectoid open unit disc. That is, 
$$\D_\infty=\text{open subspace of Spa}(C\langle q^{1/p^\infty}\rangle,\O_C\langle q^{1/p^\infty}\rangle) \text{ defined by }|q|<1.$$

 Define the left action of $\Z_p$ on $\underline{\mathrm{GL}_2(\Z_p)}\times \D_\infty$ by $h\cdot(\gamma,q^{1/p^n})= (\left(\begin{smallmatrix}
        1&h\\0&1
    \end{smallmatrix}\right)\gamma,\zeta_{p^n}^{h/e}q^{1/p^n})$. By \cite[Theorem 1.1]{heuer2022cusps},  there is a Cartesian diagram
    \[\begin{tikzcd}
        \Z_p\backslash(\underline{\mathrm{GL}_2(\Z_p)}\times \D_\infty)\arrow[r]\arrow[d,hook] &\D\arrow[d,hook]\\
        \X_{K^p}\arrow[r]&\X
    \end{tikzcd}\]
    for which the left map is a $\text{GL}_2(\Z_p)$-equivariant open immersion for the natural right action on $\text{GL}_2(\Z_p)$ and $\X_{K^p}$. Moreover, the Hodge-Tate period map $\pi_{\mathrm{HT}}:\X_{K^p}\ra \P^1$ restricts to the map
    $$\Z_p\backslash(\underline{\mathrm{GL}_2(\Z_p)}\times \D_\infty)\ra \underline{\P^1(\Z_p)},\quad (\left(\begin{smallmatrix}
        a&b\\c&d
    \end{smallmatrix}\right),q)\mapsto (c,d).$$

\begin{remark}
    As explained in \cite{colmez2021factorisation}, the structure of cusps at infinite level is also related to the Kirillov model of the completed cohomology.
\end{remark}

%\section*{Main Results}
Let $\mathcal{D}_{\Gamma(p^\infty)}:=\Z_p\backslash(\underline{\mathrm{GL}_2(\Z_p)}\times \D_\infty)$ and $\O_\infty:=\pi_{\text{HT}*}\mathcal{O}_{\mathcal{D}_{\Gamma(p^\infty)}}$, viewed as a sheaf on $\underline{\P^1(\Z_p)}$. For any $U\subset \P^1(\Z_p)$ an open subset, let   $V$ be the inverse image of $U$ under the map $\GL_2(\Q_p)\ra \P^1(\Q_p), \left(\begin{smallmatrix}
        a&b\\c&d
\end{smallmatrix}\right)\mapsto (c:d)$.  Then the section $\O_\infty(U)$ is given by 
\begin{align*}
    \O_\infty(U)=\{f:V\ra \O(\D_\infty)
    \text{ such that }f \text{ is continuous and }\\f(\ma{a&b\\c&d})=h\cdot f(\ma{a+hc&b+hd\\c&d}) \text{ for any }h\in \Z_p, \abcd\in V \}
\end{align*}
We denote $\O^\la_\infty\subset \O_\infty$ as the subsheaf of $\GL_2(\Q_p)$-locally analytic sections. Then we show that there is also a horizontal action $\theta_\h$ on $\O^\la_\infty$. Moreover, we define the intertwining operator
\[I_{k-1}: \O_{\infty}^{\la,(1-k,0)}\ra \O_{\infty}^{\la,(1,-k)}(k)\]
and provide an explicit formula for this operator in our setting.

\begin{prop} 
Let $U$ be a open subset of $\P^1(\Q_p)$ with $(0,1)\notin U$ and let $V$ be the inverse image of $U$ under the map $\GL_2(\Q_p)\ra \P^1(\Q_p)$. Then for any $f:V\ra \O(\D_\infty)$ in $\O_{\infty}^{\la,(1-k,0)}(U)$, $$I_{k-1}(f)=\frac{(ad-bc)^{k}}{e^kc^{2k}}(a\pa _b+c\pa_d)^{k}(q\frac{d}{dq})^{k}f.$$
\end{prop}
For the Fontaine operator, we can also construct a map
\[N_k:\O_{\infty}^{\la,(1-k,0)}\ra \O_{\infty}^{\la,(1,-k)}(k)\]
using $p$-adic Hodge theory. We have the following computation.
\begin{prop}
    For any $f:V\ra \O(\D_\infty)$ in $\O_\infty^{\la,(1-k,0)}(U)$, $$N_k(f)=(-1)^{k}\frac{1}{e^kk((k-1)!)^2}(a\pa_c+b\pa_d)^{k}(q\frac{d}{dq})^{k}f.$$
\end{prop}

Combining these results with additional computations we get 
\begin{thm}
    As maps from $\O_\infty^{\la,(1-k,0)}(U)$ to $\O_\infty^{\la,(1,-k)}(U)(k)$, we have  \[N_k=\frac{1}{k((k-1)!)^2}I_{k-1}.\]
\end{thm}

Using the $q$-expansion principle at infinite level proved in \cite{heuer2022cusps}, we can show that these two operators coincide at the level of cohomology. Note that the maps $N_k$ and $I_{k-1}$ from $\O_{K^p}^{\la,(1-k,0)}$ to $\O_{K^p}^{\la,(1,-k)}(k)$ induce maps $I_{k-1}^1$ and $N_k^1$ from $H^1(\fl ,\O_{K^p}^{\la,(1-k,0)})$ to $H^1(\fl, \O_{K^p}^{\la,(1,-k)})(k)$.

\begin{thm}
    As maps from $H^1(\fl ,\O_{K^p}^{\la,(1-k,0)})$ to $H^1(\fl, \O_{K^p}^{\la,(1,-k)})(k)$, we have  $$N_k^1=\frac{1}{k((k-1)!)^2}I^1_{k-1}. $$
\end{thm}

This result without the explicit constant was first established by Pan in \cite[\S 6]{pan2022}. As a corollary, we have the following description of $\tilde{H}^1(K^p,C\otimes_{\Q_p}E)[\lambda]^\la$, which was first proved in \cite[Theorem 7.2.2]{pan2022} and allows us to study the explicit structure of $\tilde{H}^1(K^p,C\otimes_{\Q_p}E)[\lambda]^\la$.

\begin{cor}
    There is a natural $\GL_2(\Q_p)$-equivariant isomorphism 
    \[\tilde{H}^1(K^p,C\otimes_{\Q_p}E)[\lambda]^\la_0\cong (\ker I^1_{k-1}\otimes_{\Q_p} E)[\lambda],\]
    where $\tilde{H}^1(K^p,C\otimes_{\Q_p}E)[\lambda]^\la_0$ is the Hodge-Tate weight $0$ part of $\tilde{H}^1(K^p,C\otimes_{\Q_p}E)[\lambda]^\la$.
\end{cor}

\subsection*{Acknowledgment}
I am deeply grateful to my advisor, Liang Xiao, for his continuous guidance, patience and support. I would like to express my sincere thanks to Lue Pan for helpful discussions during my visit to Princeton University, where the main ideas of this paper were initiated. I am also grateful to Yiwen Ding,  Ruochuan Liu and Benchao Su  for their valuable comments and suggestions.

\section{Preliminaries}\label{fon}
In this section we recall the definition of the Fontaine operator and its relation to de Rham representations. This is originally discovered in \cite{fontaine2004arithmetique}  and generalized in \cite{pan2022}. The main reference is Section 6.1 of \cite{pan2022}.

Let $K$ be a finite extension of $\Q_p$. We denote by $K_\infty\subset C$ the maximal $\Z_p$-extension of $K$ in $K(\mu_{p^\infty})$. For a Banach $C$-module $W$ with a continuous semilinear action of $G_K$, let $W^K\subset W$ be the subspace of $\Gal_{K_\infty}$-fixed, $\Gal_K$-analytic vectors in $W$. This is a $K$-Banach space and there is a natural map
$$\varphi^K:C\widehat{\otimes}_KW^K\ra W.$$
By definition, there is a natural action of the Lie algebra $\mathrm{Lie}(\Gal(K_\infty/K))\cong \Q_p$ on $W^K$. The action of $1\in \Q_p\cong \mathrm{Lie}(\Gal(K_\infty/K))$ is called the Sen operator.

\begin{defn}
    We say $W$ is Hodge-Tate of weight $\{w_1,\dots, w_n\}\subset \Z$ if there exists a finite extension $L$ of $K$ in $C$ such that $\varphi^L:C\widehat{\otimes}_LW^L\ra W$ is an isomorphism, and the action of the Sen operator on $W^L$ is semi-simple with eigenvalues $-w_1,\dots,-w_n$.
\end{defn}

Now let $W$ be a flat Banach $B_{\text{dR}}^+/(t^{k+1})$-module equipped with a continuous semilinear action of $\Gal_K$, and assume that $W$ is Hodge--Tate of weights  $0,k$ for some integer $k>0$, meaning that $W/tW$ is Hodge--Tate of weights $0,k$. Equip $W$ with the natural $t$-adic filtration and let $W_i=\text{gr}^iW=t^iW/t^{i+1}W$ ($0\leq i\leq k)$. Since $W$ is flat over $B_{\text{dR}}^+/(t^{k+1})$, the subspace $t^iW\subset W$ is a closed subspace because it can be identified with the kernel of $W\xrightarrow{\times t^{k-i}}W$. Hence there is a natural isomorphism  $W_i\cong W_0(i)$. Let $W_0=W_{0,0}\oplus W_{0,-k}$ be the Hodge--Tate decomposition with the Sen operator acting by $0$ and $-k$ respectively. Then $W_i$ is also Hodge--Tate with $W_i=W_{i,0}\oplus W_{i,-k}$ such that $W_{i,j}\cong W_{0,j}(i)$ for $j\in \{0,-k\}$. By \cite[Lemma 6.1.11]{pan2022} we get $W_0=C\widehat{\otimes}_KW_0^K$ and $\text{gr}^iW^K=W_i^K$. 

We denote by $\nabla\in \text{End}_K(W^K)$ the Sen operator on $W^K$ defined by the action of $1\in \Q_p\cong \text{Lie}(\Gal(K_\infty/K))$. Let $E_0(W^K)$ be the generalized eigenspace associated with the eigenvalue $0$. Then there exists an exact sequence
$$0\ra W_{k,-k}^K\ra E_0(W^K)\ra W^K_{0,0}\ra 0.$$
The action of $\nabla$ thus induces a $K$-linear map
$$N_W^K:W_{0,0}^K\ra W_{k,-k}^K\cong W_{0,-k}^K(k),$$
which can be extended to a $C$-linear map
$$N_W:W_{0,0}\ra W_{k,-k}\cong W_{0,-k}(k).$$
The map $N_W$ is called the Fontaine operator associated with $W$. In the classical case, we have the following result.
\begin{thm}\label{N=0}
	\textup{(\cite[Theorem 6.1.16]{pan2022})} Let $V$ be a two-dimensional continuous representation of $\Gal_{\Q_p}$ over $E$, where $E$ is a finite extension of $\Q_p$. Let $
    W=V\otimes_{\Q_p}B^+_{\text{dR}}/(t^{k+1})$. Then $N_W=0$ if and only if $V$ is a de Rham representation.
\end{thm}
\begin{pf}
    Recall that $V$ is de Rham if and only if $\dim_E D_{\dr}(V)=\dim_E V=2$. As $V$ is of Hodge--Tate weights $0,k$,  $(V\otimes_{\Q_p} C(l))^{\mathrm{Gal}_{\Q_p}}=0$ when $l\ge k+1$. Thus the natural map $(V\otimes_{\Q_p}B_{\dr}^+)^{\mathrm{Gal}_{\Q_p}}\to V\otimes_{\Q_p}B_{\dr}^+/(t^{k+1})=W$ is injective. Hence it induces an inclusion 
\[
    D_{\dr}(V)\subset E_0(W^{\Q_p}).
\]
Note that $E_0(W^{\Q_p})$ is a $2$-dimensional $E$-vector space. If $\dim_ED_{\dr}(V)=2$, then this inclusion becomes an isomorphism, so that $N_{W}=0$. 

Conversely, if $N_{W}=0$, then $E_0(W^{\Q_p})$ is fixed by $\mathrm{Gal}_{\Q_p}$. Since $H^1(\mathrm{Gal}_{\Q_p}, V\otimes_{\Q_p} C(l))=0$ when $l\geq k+1$, the natural map $V \otimes_{\Q_p}(B^+_{\text{dR}}/(t^{l}))^{\mathrm{Gal}_{\Q_p}}\ra W^{\mathrm{Gal}_{\Q_p}}$ is an isomorphism when $l\geq k+1$. Taking the limit as $l \ra\infty$, we conclude that $\dim_E D_{\dr}(V)=2$.
\end{pf}

\section{Explicit computations at cusps}
Let $p$ be a prime and $C=\widehat{\overline{\Q}}_p$. Let $N\in \Z$ be coprime to $p$ and $K^p\subset \mathrm{GL}_2(\A_f^p)$ be a fixed tame level such that $\Gamma(N)\subset K^p$. For an open compact subgroup $K_p\subset \mathrm{GL}_2(\Q_p)$, let $X_{K^pK_p}$ be the modular curve over $\Q$ of level $K^pK_p$ and 
$\X_{K^pK_p}$ be the adic space associated to $X_{K^pK_p,C}$. By \cite[Theorem III.1.2]{Sch15}, there is a perfectoid space $\X_{K^p}$ over $C$ such that
$$\mathcal{X}_{K^p}\sim \varprojlim_{K_p\subset \GL_2(\Q_p)}\mathcal{X}_{K^pK_p}.$$ Moreover, there is a $\GL_2(\Q_p)$-equivariant Hodge--Tate period map
$$\pi_{\HT}:\mathcal{X}_{K^p}\ra \mathscr{F}\ell\cong \mathbb{P}^1$$
where $\fl$ denotes the adic space over $C$ associated with the flag variety for $\GL_2$. Let 
$\mathcal{O}_{K^p}:=\pi_{\HT,*}\O_{\mathcal{X}_{K^p}}$ and $\O_{K^p}^\la\subset \O_{K^p}$ be the subsheaf of $\GL_2(\Q_p)$-locally analytic sections.

Set $\X=\X_{K^p\mathrm{GL}_2(\Z_p)}$ for simplicity. For each cusp $x$ of $\X$, it corresponds to a $K^p$-level structure on the Tate curve $T(q^e)$ over $\O_C((q))$ for some positive integer $e|N$. The analytic Tate curve parameter space induces a canonical open immersion $\D\hookrightarrow \X$ sending the origin to $x$, where $\D$ is the adic open unit disc. Let $\D_\infty$ be the perfectoid open unit disc. That is, 
$$\D_\infty=\text{the open subspace of Spa}(C\langle q^{1/p^\infty}\rangle,\O_C\langle q^{1/p^\infty}\rangle) \text{ defined by }|q|<1.$$
From \cite[Theorem 1.1]{heuer2022cusps}, we have the following description of cusps at infinite level:
\begin{thm}\label{inf}
    Let $x$ be a cusp of $\X$. Define the left action of $\Z_p$ on $\underline{\mathrm{GL}_2(\Z_p)}\times \D_\infty$ by $h\cdot(\gamma,q^{1/p^n})= (\left(\begin{smallmatrix}
        1&h\\0&1
    \end{smallmatrix}\right)\gamma,\zeta_{p^n}^{h/e}q^{1/p^n})$. Then there is a Cartesian diagram
    \[\begin{tikzcd}
        \Z_p\backslash(\underline{\mathrm{GL}_2(\Z_p)}\times \D_\infty)\arrow[r]\arrow[d,hook] &\D\arrow[d,hook]\\
        \X_{K^p}\arrow[r]&\X
    \end{tikzcd}\]
    in which the left map is a $\text{GL}_2(\Z_p)$-equivariant open immersion with respect to the natural right action on $\text{GL}_2(\Z_p)$ and $\X_{K^p}$. Moreover, the Hodge-Tate period map $\pi_{\mathrm{HT}}:\X_{K^p}\ra \P^1$ restricts to the map
    $$\Z_p\backslash(\underline{\mathrm{GL}_2(\Z_p)}\times \D_\infty)\ra \underline{\P^1(\Z_p)},\quad (\left(\begin{smallmatrix}
        a&b\\c&d
    \end{smallmatrix}\right),q)\mapsto (c,d)$$
\end{thm}
\begin{pf}
    See \cite[Theorem 3.22]{heuer2022cusps}. Note that we use the right action of $\text{GL}_2(\Z_p)$ on the modular curve instead of the left action, so the notation is slightly different. In our notation, a point $(\left(\begin{smallmatrix}
        a&b\\c&d
    \end{smallmatrix}\right),q)\in \Z_p\backslash(\underline{\mathrm{GL}_2(\Z_p)}\times \D_\infty)$ corresponds to a Tate curve $E=T(q^e)$ equipped with an isomorphism
    % $$\alpha:\Z_p^{\oplus 2}\cong T_p(E),\quad (1,0)\mapsto \zeta_{p^\infty}^aq^{c/p^\infty}, (0,1)\mapsto \zeta_{p^\infty}^bq^{d/p^\infty}.$$ 
    \[\alpha:T_p(E)\cong \Z_p^{\oplus 2} ,\quad q^{\frac{e}{p^\infty}}\mapsto (a,b), \zeta_{p^\infty}\mapsto (c,d) .\]
    % Since the canonical subgroup corresponds to the subgroup of $\Z_p^{\oplus 2}$ generated by $\alpha^{-1}(\zeta_{p^\infty})=(d,-c)$, it maps to $(c:d)$ via $\pi_{\mathrm{HT}}$.
    Since the canonical subgroup of $T_p(E)$ is $\zeta_{p^\infty}$, it maps to $(c,d)$ via $\pi_{\mathrm{HT}}$.
\end{pf}
~\\

Set $\D_{\Gamma(p^\infty)}:=\Z_p\backslash(\underline{\mathrm{GL}_2(\Z_p)}\times \D_\infty)$. The global sections of $\D_{\infty}$ is given by 
\begin{align*}
    \O(\D_{\infty})=\{\sum_{n\in \Z[\frac{1}{p}]_{\geq 0}} a_nq^n\in C[[q^{1/p^\infty}]]
    \text{ such that }|a_n|q^n\ra 0 \text{ for all } 0\leq q<1 \\\text{and } |a_n|\ra 0 \text{ on bounded intervals} \}
\end{align*}
where the second condition means that for any $\delta>0$ and any bounded interval $I\subset \Z[\frac{1}{p}]_{\geq 0}$ there are only finitely many $n\in I$ such that $|a_n|>\delta$. We also have that $\O^+(\D_\infty)=\O_C[[q^{1/p^\infty}]]$ is the $(p,q)$-adic completion of $\varinjlim_{n\in \N}\O_C[[q^{1/p^n}]]$.

Let $\O_\infty:=\pi_{\HT*}\O_{\D_{\Gamma(p^\infty)}}$ be a sheaf on $\underline{\P^1(\Z_p)}$. 
For any $U\subset \P^1(\Z_p)$ an open subset, let  $V_\infty=\pi_{\mathrm{HT}}^{-1}(U)=\Z_p\backslash(\underline{V}\times \D_\infty)$, where $V$ is the inverse image of $U$ under the map $\GL_2(\Q_p)\ra \P^1(\Q_p), \left(\begin{smallmatrix}
        a&b\\c&d
\end{smallmatrix}\right)\mapsto (c:d)$.  Then the section $\O_\infty(U)$ is given by 
\begin{align*}
    \O_\infty(U)=\{f:V\ra \O(\D_\infty)
    \text{ such that }f \text{ is continuous and }\\f(\ma{a&b\\c&d})=h\cdot f(\ma{a+hc&b+hd\\c&d}) \text{ for any }h\in \Z_p, \abcd\in V \}
\end{align*}

\iffalse
\begin{align*}
    \O(\D_{\Gamma(p^\infty)})=\{f:\mathrm{GL}_2(\Z_p)\ra \O(\D_\infty)
    \text{ such that }f \text{ is continuous and }\\f(a,b,c,d)=h\cdot f(a+hc,b+hd,c,d) \text{ for any }h\in \Z_p \}
\end{align*}
\fi

Here the action of $\Z_p$ on $\O(\D_\infty)$ is given by $h\cdot q^{1/p^n}=\zeta_{p^n}^{h/e}q^{1/p^n}$. 
Since $V$ is open, there exists $n\in \N$ such that $V$ is stable under the natural  right $\Gamma(p^n)$ action, which induces a left $\Gamma(p^n)$ action on $\O_\infty(U)$. We denote by $\O^\la_\infty(U)$ the subspace of  locally analytic vectors with respect to this action. For $m\in \N$, let $\D_m$ be the $(\Z/p^m\Z)$-Galois cover of $\D$ with global sections given by 
$$\O(\D_m)=\{\sum_{n\in \frac{1}{p^m}\Z}a_nq^n \in C[[q^{1/p^m}]]
    \text{ such that }|a_n|q^n\ra 0 \text{ for all } 0\leq q<1\}.$$
Let $\O^\sm(\D_\infty)=\varinjlim_{m\in \N}\O(\D_m)$. It consists of the smooth functions in $\O(\D_\infty)$ with respect to the action of $\Z_p$.

\begin{prop}\label{la}
    Let $U,V,V_\infty$ be as above. For any $f:V\ra \O(\D_\infty)$ in $\O_\infty^\la(U)$, assume that $f$ is $\Gamma(p^n)$-analytic for some $n\in \N$. Then
    the image of $f$ lies in $\O(\D_n)$.
\end{prop}
\begin{pf}
    For any $g\in \mathrm{GL}_2(\Z_p)$, 
    since the function $g\cdot f:g'\mapsto f(g'g)$ defines a $\Gamma(p^n)$-analytic section on $\Z_p\backslash(\underline{Vg^{-1}}\times \D_\infty)$
    and $(g\cdot f)(\ma{1&0\\0&1})=f(g)$, it suffices to prove that $f(\ma{1&0\\0&1})\in \O(\D_n)$. 
    For any $m\in \N$, define the Tate trace $\mathrm{tr}_m:\O(\D_\infty)\ra \O(\D_m)$ by sending $\sum_{n\in \Z[\frac{1}{p}]_{\geq 0}} a_nq^n$ to $\sum_{n\in \frac{1}{p^m}\Z}a_nq^n$. Since this map is continuous and $\Z_p$-equivariant, $f_m:=\mathrm{tr}_m(f)$ is a $\Gamma(p^n)$-analytic function taking values in $\O(\D_m)$. Since $\O(\D_m)$ is invariant under the action of $p^m\Z_p$, we have $f_m(\ma{a&b\\c&d})=h\cdot f_m(\ma{a+hc&b+hd\\c&d})=f_m(\ma{a+hc&b+hd\\c&d})$ for any $h\in p^m\Z_p$. Setting $\ma{a&b\\c&d}=\ma{1&0\\0&1}$, we obtain  $f_m(\ma{1&0\\0&1})=f_m(\ma{1&h\\0&1})$ for any $h\in p^m\Z_p$. Since  $f_m$ is $\Gamma(p^n)$-analytic, this implies that $f_m(\ma{1&0\\0&1})=f_m(\ma{1&h\\0&1})$ for any $h\in p^n\Z_p$. Thus $f_m(\ma{1&0\\0&1})=f_m(\ma{1&h\\0&1})=h\cdot f_m(\ma{1&h\\0&1})$ for any $h\in p^n\Z_p\Rightarrow f_m(\ma{1&0\\0&1})\in \O(D_n)$. Since this is true for all $m$, we conclude that $f(\ma{1&0\\0&1})\in \O(D_n)$.
\end{pf}
~\\

From the definition of $\O_\infty$, there is a natural action of $\O_{\underline{\P^1(\Z_p)}}$ on $\O_\infty$. Consider the tautological exact sequence on $\P^1$
\[0\ra \O_{\P^1}(-1)\ra \O_{\P^1}^{\oplus 2}\ra \O_{\P^1}(1)\ra 0.\]
Let $e_1, e_2$ be the global section of $\O_{\underline{\P^1(\Z_p)}}(1)$ defined by the image of $(1,0),(0,1)$.
If $U\subset \P^1(\Z_p)$ is a open subset not containing $\infty$, then $e_1$ is invertible on $U$ and we denote  the standard coordinate by $x=\frac{e_2}{e_1}$. From Theorem \ref{inf} we know that the action of $\mathcal{O}(U)$ is given by $x\cdot f=\frac{d}{c}f$. Let $\gl_2(\Q_p)$ be the Lie algebra of $\GL_2(\Q_p)$. Then $\O^\la_{\infty}$ carries an action of $\gl_2(\O_{\underline{\P^1(\Z_p)}}):=\gl_2(\Q_p)\otimes_{\Q_p}\O_{\underline{\P^1(\Z_p)}}$. Let 
\begin{align*}
    \mathfrak{b}^0=\{f\in \gl_2(\O_{\underline{\P^1(\Z_p)}})\ |\ f_x\in \b_x \text{ for all } x\in \P^1(\Z_p)\}\\
    \mathfrak{n}^0=\{f\in \gl_2(\O_{\underline{\P^1(\Z_p)}})\ |\ f_x\in \n_x \text{ for all } x\in \P^1(\Z_p)\}
\end{align*}
where $\b_x$ and $\n_x$ are the Borel subalgebra and nilpotent subalgebra corresponding to $x$.

\begin{prop}\label{n0}
    $\mathfrak{n}^0$ acts trivially on $\O^\la_{\infty}$. 
\end{prop}
 \begin{pf}
    Let $U\subset \P^1(\Z_p)$ be an open subset and $V_\infty=\pi_{\mathrm{HT}}^{-1}(U)=\Z_p\backslash(\underline{V}\times \D_\infty)$. Assume without loss of generality that $\infty\notin U$. Then $\n^0$ has a basis given by $\ma{x&x^2\\-1&-x}$. Then for any $f:V\ra \O(\D_\infty)$ in $\O_\infty^\la(U)$, calculating via matrix multiplication we know that the action is given by
    \begin{align*}
        \ma{x&x^2\\-1&-x}\cdot f&=(\frac{d}{c}(a\partial_a+c\partial_c)+(\frac{d}{c})^2(a\partial_b+c\partial_d)-(b\partial_a+d\partial_c)-\frac{d}{c}(b\partial_b+d\partial_d)f\\
        &=\frac{ad-bc}{c^2}(c\partial_a+d\partial_b)f.
    \end{align*}
    By Proposition \ref{la} we know that there exists $n\in \N$ such that the image of $f$ is contained in $\O(\D_n)$. Thus
    $$f(\abcd)=h\cdot f(\ma{a+hc&b+hd\\c&d})=f(\ma{a+hc&b+hd\\c&d})$$
    for any $h\in 1+p^n\Z_p$. Differentiating with respect to $h$ and setting $h=0$, we obtain $(c\partial_a+d\partial_b)f=0$. Therefore $ \ma{x&x^2\\-1&-x}\cdot f=0$.
\end{pf}
~\\

Therefore, there is an action of $\b^0/\n^0\cong \h\otimes_{\Q_p}\O_{\underline{\P^1(\Z_p)}}$ on $\O^\la_{\infty}$, where $\h$ is the  Lie algebra of the $2\times 2$ diagonal matrix. Let $\theta_\h$ be the horizontal action of $\h$ on $\O^\la_{\infty}$ induced by the natural embedding $\h\hookrightarrow \b^0/\n^0$.

Now we study the Galois action on  $\O_{\infty}$. Since the Tate curve and the open immersion $\D\hookrightarrow\X$ can be defined over $K:=\Q_p(\zeta_N)$, the action of $\gal_K$ on $\O(\D_\infty)$ corresponds to the action on the coefficients. Since  $\mu_{p^\infty}$ corresponds to $(c,d)$ under the moduli problem and has Hodge-Tate weight $-1$, the action of $\gal_K$ on $\O_{\D_{\infty}}$ is given by 
$$(\sigma f)(\abcd)=\sigma(f(\ma{a&b\\ \chi(\sigma)c& \chi(\sigma)d}))$$
for any $\sigma\in \gal_K, f:V\ra \O(\D_\infty)$ in $\O_\infty(U)$ with $V_\infty=\pi_\HT^{-1}(U)$, where $\chi:\gal_K\ra \Q_p^\times$ is the cyclotomic character. Thus, if we restrict it to $\O_\infty^\la(U)$, it has a unique Sen operator given by $c\partial_c+d\partial_d$.

\begin{thm}\label{sen}
    $\theta_\h(\ma{0&0\\0&1})$ is the Sen operator on $\O_\infty
    ^\la(U)$.
\end{thm}
\begin{pf}
    It suffices to prove that $\theta_\h(\ma{0&0\\0&1})=c\partial_c+d\partial_d$. Assume without loss of generality that $\infty\notin U$. Then a direct computation (for details, see \cite[5.1.1]{panI}) shows that  $\theta_\h(\ma{a&0\\0&d})$ acts on $\O_\infty
    ^\la(U)$ as $\ma{d&(d-a)x\\0&a}\in \gl_2(\O_{\underline{\P^1(\Z_p)}})$. Thus 
    \begin{align*}
        \theta_\h(\ma{0&0\\0&1})\cdot f&= \ma{1&x\\0&0}\cdot f=((a\partial_a+c\partial_c)+\frac{d}{c}(a\pa_b+c\pa_d)f\\&=(c\partial_c+d\pa_d))f
    \end{align*}
    where the last equality follows from the previously proved identity $(c\pa_a+d\pa_b)f=0$ in Proposition \ref{n0}.
\end{pf}
~\\

For $n_1,n_2\in \Z$, let $\chi=(n_1,n_2)$ be the character of $\h$ sending $\ma{x&0\\0&y}$ to $n_1x+n_2y$. Let $\O_{\infty}^{\la,(n_1,n_2)}$ denote the weight-$\chi$ subsheaf of $\O_{\infty}^{\la}$ under $\theta_\h$ and assume that $k=n_2-n_1\geq 0$. As explained in \S 4 of \cite{pan2022}, there is a $(k+1)$-th differential map along modular curves at finite level
$$d^{k+1}:\O_{\infty}^{\la,(n_1,n_2)}\ra \O_{\infty}^{\la,(n_1,n_2)}\otimes_{\O_{\infty}^\sm}(\Omega^1_{\infty}(\mathcal{C})^\sm)^{\otimes k+1}$$
sending $f:V\ra \O(\D_\infty)$ in $\O_{\infty}^{\la,(n_1,n_2)}(U)$ to $(q^e\frac{d}{dq^e})^{k+1}f\otimes (\frac{dq}{q})^{k+1}=(\frac{q}{e}\frac{d}{dq})^{k+1}f\otimes (\frac{dq}{q})^{k+1}$, where 
$\O_{\infty}^\sm$ denotes the smooth vectors in $\O_{\infty}$ and 
$\Omega^1_{\infty}(\mathcal{C})^\sm$ is the pushforward of the sheaf of differentials with log pole at $q=0$, which is a free $\O_{\infty}^\sm$-module of rank $1$ generated by $\frac{dq}{q}$. Note that in our notation the Tate curve corresponds to $T(q^e)$. Moreover, the $(k+1)$-th differential map along the flag variety induces a map
$$\overline{d}^{k+1}:\O_{\infty}^{\la,(n_1,n_2)}\ra \O_{\infty}^{\la,(n_1,n_2)}\otimes_{\O_{\P^1}}(\Omega^1_{\P^1})^{\otimes k+1}.$$
For $f:V\ra \O(\D_\infty)$ in $\O_{\infty}^{\la,(n_1,n_2)}(U)$ with $e_1\neq 0$ in $U$, it is defined by $\overline{f}^{k+1}(s)=(u^+)^{k+1}(f)\otimes (dx)^{k+1}=(a\pa_b+c\pa_d)^{k+1}f\otimes (dx)^{k+1}$.

As in \cite[4.2.3]{pan2022}, we have the following lemma.
\begin{lem}
    There is a $\GL_2(\Q_p)$-equivariant isomorphism
$$(\Omega^1_{\infty}(\mathcal{C})^\sm)^{\otimes k+1}\otimes_{\O_{\infty}^\sm}\O_{\infty}^{\la,(n_1,n_2)}\otimes_{\O_{\P^1}}(\Omega^1_{\P^1})^{\otimes k+1}\cong \O_{\infty}^{\la,(n_2+1,n_1-1)}(k+1)$$
given by sending $(\frac{dq}{q})^{k+1}\otimes f\otimes (dx)^{k+1}$ to $\frac{(ad-bc)^{k+1}}{c^{2(k+1)}}f$ on the locus where $e_1$ is invertible, and sending $(\frac{dq}{q})^{k+1}\otimes f\otimes (dy)^{k+1}$ to $(-1)^{k+1}\frac{(ad-bc)^{k+1}}{d^{2(k+1)}}f$ on the locus where $e_2$ is invertible and $y:=\frac{1}{x}$.
\end{lem}
\begin{pf}
    On the locus where $e_1,e_2$ are both invertible, we have $dy=-\frac{1}{x^2}dx$. Thus, by the first formula, it sends $(\frac{dq}{q})^{k+1}\otimes f\otimes (dy)^{k+1}$ to
    \[(-1)^{k+1}\frac{(ad-bc)^{k+1}}{c^{2(k+1)}x^{2(k+1)}}f= (-1)^{k+1}\frac{(ad-bc)^{k+1}}{d^{2(k+1)}}f.\]
    This shows that the map is well-defined. It is clear that the map is an isomorphism. Thus 
    it remains to prove that this map is $\GL_2(\Q_p)$-equivariant. Consider the locus where $e_1$ is invertible, 
    for $\gamma=\begin{pmatrix}
        a_0&b_0\\c_0&d_0
    \end{pmatrix} \in \GL_2(\Q_p)$, we have 
    \[\gamma((\frac{dq}{q})^{k+1}\otimes f\otimes (dx)^{k+1})=(\frac{dq}{q})^{k+1}\otimes \gamma f\otimes (d\frac{b_0+d_0x}{a_0+c_0x})^{k+1}= (\frac{dq}{q})^{k+1}\otimes \gamma f\otimes \frac{(a_0d_0-b_0c_0)^{k+1}}{(a_0+c_0x)^{2(k+1)}}(dx)^{k+1}.\]
    Consequently, the image of $\gamma((\frac{dq}{q})^{k+1}\otimes f\otimes (dx)^{k+1})$ is  
    \[\frac{(ad-bc)^{k+1}}{c^{2(k+1)}}\cdot \frac{(a_0d_0-b_0c_0)^{k+1}}{(a_0+c_0\frac{d}{c})^{2(k+1)}}\gamma f=\frac{(ad-bc)^{k+1}(a_0d_0-b_0c_0)^{k+1}}{(a_0c+c_0d)^{2(k+1)}} \gamma f=\gamma (\frac{(ad-bc)^{k+1}}{c^{2(k+1)}}f).\]
     On the locus where $e_2$ is invertible the computation is analogous. Therefore, the map is a $\GL_2(\Q_p)$-equivariant isomorphism.
\end{pf}
~\\

Combining all these maps, we obtain the following proposition.
\begin{prop}\label{I}
    There is a $GL_2(\Q_p)$-equivariant intertwining map
    $$I_k=d^{k+1}\circ \overline{d}^{k+1}= \overline{d}^{k+1}\circ d^{k+1}:\O_{\infty}^{\la,(n_1,n_2)}\ra \O_{\infty}^{\la,(n_2+1,n_1-1)}(k+1)$$
sending $f:V\ra \O(\D_\infty)$ in $\O_{\infty}^{\la,(n_1,n_2)}(U)$ to $\frac{(ad-bc)^{k+1}}{e^{k+1}c^{2(k+1)}}(q\frac{d}{dq})^{k+1}(a\pa _b+c\pa_d)^{k+1}f$.
\end{prop}

Next, we consider the sheaf $\bdr$ and the induced Fontaine operator as in \S 6.2 of \cite{pan2022}. As previously explained, $\D_{\Gamma(p^\infty)}$ and the Hodge-Tate map can be defined over a finite extension of $\Q_p$, thus it makes sense to consider the sheaf $\B_{\dr,\D_{\Gamma(p^\infty)}}^+$. Let $\B_{\dr,\infty}^+:=\pi_{\HT,*} \B_{\dr,\D_{\Gamma(p^\infty)}}^+$ be a sheaf on $\underline{\P^1(\Z_p)}$.  For any open subset $U\subset \P^1(\Z_p)$ with $\pi^{-1}(U)=\Z_p\backslash (\underline{V}\times \D_\infty)$, we have
\begin{align*}
    \B_{\dr,\infty}^+(U)=\{f:V\ra \bdr^+(\D_\infty)
    \text{ such that }f \text{ is continuous and }\\f(\ma{a&b\\c&d})=h\cdot f(\ma{a+hc&b+hd\\c&d}) \text{ for any }h\in \Z_p, \abcd\in V \}.
\end{align*}
The sheaf $\B_{\dr,\infty}^+$ admits a natural decreasing filtration given by  $\mathrm{Fil}^k\B_{\dr,\infty}^+=t^k\B_{\dr,\infty}^+$. For any $k\geq 1$, let $\B_{\dr,\infty,k}^+:=\B_{\dr,\infty}^+/\mathrm{Fil}^k\B_{\dr,\infty}^+$
and let $\B_{\dr,\infty,k}^{+,\la}\subset \B_{\dr,\infty,k}^+$ be the subsheaf of locally analytic vectors.

For $k\geq 1,$ let
$$\Tilde{\chi}_k=\{(0,1-k),(-k,1)\}\subset \h^*$$ 
be the infinitesimal character of the $(k-1)$-th symmetric power of the dual of the standard representation. For a sheaf $\mathcal{F}$ on $\fl$ with a $\gl_2(\Q_p)$-action, denote by $\mathcal{F}^{\Tilde{\chi}_k}$ the subsheaf on which $Z(U(\gl_2(\Q_p)))$ acts via $\Tilde{\chi}_k$. 
It follows from the relation between $\theta_\h$ and the infinitesimal character \cite[Corollary 4.2.8]{panI} that on $\O_\infty^{\la,\Tilde{\chi}_k}$, we have a natural decomposition
$$\O_\infty^{\la,\Tilde{\chi}_k}=\O_\infty^{\la,(1-k,0)}\oplus \O_\infty^{\la,(1,-k)}.$$

Let $U\subset \P^1(\Z_p)$ be an open subset. Then $\Gal_K$ acts naturally on $\B_{\dr,\infty,k+1}^{+,\la,\Tilde{\chi}_k}(U)$ for some finite extension $K$ of $\Q_p$. There is a $t$-adic filtration on $\B_{\dr,\infty,k+1}^{+,\la,\Tilde{\chi}_k}(U)$ whose $i$-th graded part ($0\leq i\leq k$) is given by
$$\mathrm{gr}^i(B_{\dr,\infty,k+1}^{+,\la,\Tilde{\chi}_k}(U))\cong \O_\infty^{\la,\Tilde{\chi}_k}(U)(i)=\O_\infty^{\la,(1-k,0)}(U)(i)\oplus \O_\infty^{\la,(1,-k)}(U)(i)$$
with Hodge-Tate weight $-i,k-i$ respectively.

We denote by $\B_{\dr,\infty,k+1}^{+,\la,\Tilde{\chi}_k}(U)^K\subset \B_{\dr,\infty,k+1}^{+,\la,\Tilde{\chi}_k}(U)$ the subspace of $\Gal_{K_\infty}$-fixed, $\Gal_K$-analytic vectors. It carries a Sen operator $\Theta$ coming from the action of  $1\in \Z_p\cong\mathrm{Lie}(\Gal(K_\infty/K)) $.
Let $E_0(\B_{\dr,\infty,k+1}^{+,\la,\Tilde{\chi}_k}(U)^K)$ be the weight $0$ part of $\B_{\dr,\infty,k+1}^{+,\la,\Tilde{\chi}_k}(U)^K$, that is, the subspace where $\Theta$ acts nilpotently. Since  $E_0(\mathrm{gr}^i(B_{\dr,\infty,k+1}^{+,\la,\Tilde{\chi}_k}(U))^K)$ equals $\O_\infty^{\la,(1-k,0)}(U)^K$ when $i=0$, equals $\O_\infty^{\la,(1,-k)}(U)(k)^K$ when $i=k$ and vanishes when $1\leq i\leq k-1$, there is a short exact sequence
$$0\ra \O_\infty^{\la,(1,-k)}(U)(k)^K \ra E_0(\B_{\dr,\infty,k+1}^{+,\la,\Tilde{\chi}_k}(U)^K)\ra \O_\infty^{\la,(1-k,0)}(U)^K\ra 0.$$
Since $\Theta$ acts trivially on the first and the third term, it induces a map
$$N_k:\O_\infty^{\la,(1-k,0)}(U)^K\ra \O_\infty^{\la,(1,-k)}(U)(k)^K$$
and it extends $C$-linearly to a map
$$N_k:\O_\infty^{\la,(1-k,0)}(U)\ra \O_\infty^{\la,(1,-k)}(U)(k).$$
We can calculate $N_k$ explicitly:
\begin{prop}\label{N}
    For any $f:V\ra \O(\D_\infty)$ in $\O_\infty^{\la,(1-k,0)}(U)$, $$N_k(f)=(-1)^{k}\frac{1}{e^kk((k-1)!)^2}((a\pa_c+b\pa_d)^{k}(q\frac{d}{dq})^{k})f.$$
\end{prop}
\begin{pf}
    It suffices to prove this for $f\in \O_\infty^{\la,(1-k,0)}(U)^K$. By definition, we may choose $\Tilde{f}\in E_0(\B_{\dr,\infty,k+1}^{+,\la,\Tilde{\chi}_k}(U)^K)$ such that $\mathrm{gr}^0(\Tilde{f})=f$, then $N_k(f)=\Theta(\Tilde{f})$. 
    
    Note that for $g=\sum_{n\in \frac{1}{p^m}\Z}a_n\Tilde{q}^n\in \B_{\dr,k+1}^{+}(\D_\infty)$ with $\Tilde{q}=[(q,q^{\frac{1}{p}},q^{\frac{1}{p^2}},\cdots)]$, the $\Z_p$ action is given by $h(g)=\sum_{n\in \frac{1}{p^m}\Z}a_n[\epsilon]^{\frac{nh}{e}}\Tilde{q}^n$ for $h\in \Z_p$. Thus the differential at $0$ is given by $\sum_{n\in \frac{1}{p^m}\Z}a_n\cdot\frac{n}{e}t\Tilde{q}^n=\frac{t}{e}\Tilde{q}\frac{d}{d\Tilde{q}}g$.
    
    Since $\Tilde{f}$ is locally analytic, it follows from Proposition \ref{la} and induction that there exists $m\in \N$ such that $\Tilde{f}(\ma{a&b\\c&d})$ is of the form $\sum_{n\in \frac{1}{p^m}\Z}a_n\Tilde{q}^n\in \B_{\dr,k+1}^{+}(\D_\infty)$ for any $\ma{a&b\\c&d}\in V$. Since 
    $\Tilde{f}(\ma{a&b\\c&d})=h\cdot \Tilde{f}(\ma{a+hc&b+hd\\c&d}) \text{ for any }h\in \Z_p, \abcd\in V$, differentiating  with respect to $h$ and setting $h = 0$,  we obtain
    $$(c\pa_a+d\pa_b+\frac{t}{e}\Tilde{q}\frac{d}{d\Tilde{q}})\Tilde{f}=0$$

    Next, we compute the infinitesimal action. The center $Z(U(\gl_2(\Q_p)))$ is generated by $Z=\ma{1&0\\0&1}$ and $\Omega=\frac{1}{2}h^2-h+2u^+u^-$ where $h=\ma{1&0\\0&-1}, u^+=\ma{0&1\\0&0}, u^-=\ma{0&0\\1&0}$. Then 
    $$Z\Tilde{f}=(a\pa_a+b\pa_b+c\pa_c+d\pa_d)\Tilde{f}$$
    and 
    \begin{align*}
        \Omega\Tilde{f}&=(\frac{1}{2}(a\pa_a+c\pa_c-b\pa_b-d\pa_d)^2-(a\pa_a+c\pa_c-b\pa_b-d\pa_d)+2(a\pa_b+c\pa_d)(b\pa_a+d\pa_c))\Tilde{f}\\
    \end{align*}
    Note that 
    \begin{align*}
        2(a\pa_b+c\pa_d)(b\pa_a+d\pa_c))\Tilde{f}=2(ab\pa_a\pa_b+a\pa_a+cd\pa_c\pa_d+c\pa_c+a\pa_cd\pa_b+b\pa_dc\pa_b)\Tilde{f}\\
        =2(ab\pa_a\pa_b+a\pa_a+cd\pa_c\pa_d+c\pa_c-a\pa_c(c\pa_a+\frac{t}{e}\Tilde{q}\frac{d}{d\Tilde{q}})-b\pa_d(d\pa_b+\frac{t}{e}\Tilde{q}\frac{d}{d\Tilde{q}}))\Tilde{f}\\
        =2(ab\pa_a\pa_b+cd\pa_c\pa_d-ac\pa_a\pa_c-bd\pa_b\pa_d-b\pa_b+c\pa_c-(a\pa_c+b\pa_d)\frac{t}{e}\Tilde{q}\frac{d}{d\Tilde{q}})\Tilde{f}.
    \end{align*}
    Thus, 
    \begin{align*}
        \Omega\Tilde{f}&=(\frac{1}{2}(a\pa_a+c\pa_c-b\pa_b-d\pa_d)^2+2(ab\pa_a\pa_b+cd\pa_c\pa_d-ac\pa_a\pa_c-bd\pa_b\pa_d)\\
        &\quad -a\pa_a-b\pa_b+c\pa_c+d\pa_d-2(a\pa_c+b\pa_d)\frac{t}{e}\Tilde{q}\frac{d}{d\Tilde{q}})\Tilde{f}\\
        &=\frac{1}{2}(a\pa_a+b\pa_b-c\pa_c-d\pa_d)^2-(a\pa_a+b\pa_b-c\pa_c-d\pa_d)-2(a\pa_c+b\pa_d)\frac{t}{e}\Tilde{q}\frac{d}{d\Tilde{q}}\Tilde{f}.
    \end{align*}
    Let $A=a\pa_a+b\pa_b$ and $B=c\pa_c+d\pa_d$. Since  $Z(U(\gl_2(\Q_p)))$ acts on $\Tilde{f}$ by $\chi_k$, we have $Z\Tilde{f}=(1-k)\Tilde{f}$ and $\Omega\Tilde{f}=\frac{1}{2}(k^2-1)\Tilde{f}$. Thus $$(A+B)\Tilde{f}=(1-k)\Tilde{f}$$ and $$(\frac{1}{2}(A-B)^2-(A-B)-2(a\pa_c+b\pa_d)\frac{t}{e}\Tilde{q}\frac{d}{d\Tilde{q}}\Tilde{f})=\frac{1}{2}(k^2-1)\Tilde{f}.$$
    Note that for any $\ma{a&b\\c&d}\in \GL_2(\Z_p)$, 
    $\Tilde{f}(\ma{a&b\\c&d})\in \B_{\dr,k+1}^+(\D_\infty)^K$ is of the form $\sum_{n\in \frac{1}{p^m}\Z}a_n\Tilde{q}^n$ with $a_n\in B_{\dr,k+1}^{K}=K_\infty[t]/t^{k+1}\cong \oplus_{i=0}^{k}K_\infty t^i$. Consequently, $\Tilde{f}$ can be written uniquely as $\Tilde{f}=\Tilde{f}_0+\Tilde{f}_1t + \cdots +\Tilde{f}_k t^k$ with $f_i(\ma{a&b\\c&d})\in K_\infty[[\Tilde{q}^{1/p^m}]]$ for any $0\leq i\leq k, \ma{a&b\\c&d}\in \mathrm{GL}_2(\Z_p)$.  
    Since $\mathrm{gr}^0(\Tilde{f})=f$, we obtain $\mathrm{gr}^0(\Tilde{f}_0)=f$. 
    % Note that $\theta$ is an isomorphism onto its image when restricted to $K_\infty[[\Tilde{q}^{1/p^m}]]\subset \B_{\dr,k+1}^+(\D_\infty)^K$, so this already determines $\Tilde{f}_0$. 
    The condition of the infinitesimal character then becomes
    \begin{equation}\label{1}
        (A+B)\Tilde{f}_i=(1-k)\Tilde{f}_i
    \end{equation}
    and
    \begin{equation}\label{2}
        (\frac{1}{2}(A-B)^2-(A-B)-\frac{1}{2}(k^2-1))\Tilde{f}_i=\frac{2}{e}(a\pa_c+b\pa_d)\Tilde{q}\frac{d}{d\Tilde{q}}\Tilde{f}_{i-1}
    \end{equation}
    for $0\leq i\leq k$, where $\Tilde{f}_{-1}=0$. 
    Since $\Tilde{f}\in E_0(\B_{\dr,\infty,k+1}^{+,\la,\Tilde{\chi}_k}(U)^K)$, we have 
    $\Theta(\Tilde{f})=\sum_{i=0}^{k}(\Theta(\Tilde{f}_i)+i\Tilde{f}_i)t^i\in t^k\B_{\dr,\infty,k+1}^{+,\la,\Tilde{\chi}_k}(U)^K$. Thus $\Theta(\Tilde{f}_i)=-i\Tilde{f}_i$ for $0\leq i\leq k-1$. Note that the action of $\gal_K$ on $\Tilde{f}_i$ is given by 
    $$(\sigma \Tilde{f}_i)(\abcd)=\sigma(\Tilde{f}_i(\ma{a&b\\ \chi(\sigma)c& \chi(\sigma)d})),\  \sigma\in \Gal_K,$$
    where $\chi:\gal_K\ra \Q_p^\times$ is the cyclotomic character. 
    Since $\Gal_K$ acts trivially on $\Tilde{q}$, $\Theta(\Tilde{f}_i)=(c\pa_c+d\pa_d)\Tilde{f}_i=B\Tilde{f}_i$.
    Thus $B\Tilde{f_i}=-i\Tilde{f}_i$ for $0\leq i\leq k-1$. Combining this with (\ref{1}), (\ref{2}), we obtain $A\Tilde{f}_i=(1-k+i)\Tilde{f}_i$ and 
    $$\Tilde{f}_i=-\frac{1}{ei(k-i)}((a\pa_c+b\pa_d)\Tilde{q}\frac{d}{d\Tilde{q}})\Tilde{f}_{i-1}$$
    for $1\leq i\leq k-1$. Thus
    $$\Tilde{f}_{k-1}=(-1)^{k-1}\frac{1}{e^{k-1}((k-1)!)^2}((a\pa_c+b\pa_d)^{k-1}(\Tilde{q}\frac{d}{d\Tilde{q}})^{k-1})\Tilde{f}_0.$$
    For $i=k$, since $\Theta^2(\Tilde{f})=0$, we have $(B+k)^2\Tilde{f}_k=0\Rightarrow B^2\Tilde{f}_k=(-2kB-k^2)\Tilde{f}_k$. Thus we obtain 
    \begin{align*}
        \frac{2}{e}(a\pa_c+b\pa_d)\Tilde{q}\frac{d}{d\Tilde{q}}\Tilde{f}_{k-1}&= (\frac{1}{2}(A-B)^2-(A-B)-\frac{1}{2}(k^2-1))\Tilde{f}_k\\
        &=(\frac{1}{2}(1-k-2B)^2-(1-k-2B)-\frac{1}{2}(k^2-1))\Tilde{f}_k\\
        &=(\frac{1}{2}((1-k)^2-4(1-k)B+4(-2kB-k^2))\\
        &\quad -(1-k-2B)-\frac{1}{2}(k^2-1))\Tilde{f}_k\\
        &=-2k(B+k)\Tilde{f}_k.
    \end{align*}
    Therefore, $$(B+k)\Tilde{f}_k=(-1)^{k}\frac{1}{e^kk((k-1)!)^2}((a\pa_c+b\pa_d)^{k}(\Tilde{q}\frac{d}{d\Tilde{q}})^{k})\Tilde{f}_0,$$
    and $$\Theta(\Tilde{f})=t^k(B+k)\Tilde{f}_k=t^k(-1)^{k}\frac{1}{e^kk((k-1)!)^2}((a\pa_c+b\pa_d)^{k}(\Tilde{q}\frac{d}{d\Tilde{q}})^{k})\Tilde{f}_0.$$
    Under the canonical isomorphism $$\mathrm{gr}^k(E_0(\B_{\dr,\infty,k+1}^{+,\la,\Tilde{\chi}_k}(U)^K))\cong \O_\infty^{\la,(1,-k)}(U)(k),$$
    this corresponds to 
    $$(-1)^{k}\frac{1}{e^kk((k-1)!)^2}((a\pa_c+b\pa_d)^{k}(q\frac{d}{dq})^{k})f.$$ 
\end{pf}

We are now ready to prove our main theorem. Recall that in Proposition \ref{I} we defined a map
$$I_{k-1}:\O_\infty^{\la,(1-k,0)}\ra \O_\infty^{\la,(1,-k)}(k).$$
\begin{thm}\label{n=ddbar}
    $N_k=\frac{1}{k((k-1)!)^2}I_{k-1}.$ 
\end{thm}
\begin{pf}
    By Proposition \ref{I} and Proposition \ref{N}, it suffices to prove that for any $f:\abcd\ra \O(\D_\infty)$ in $
    \O_\infty^{\la,(1-k,0)}(U)$, 
    \begin{equation}\label{main}
        \frac{(ad-bc)^k}{c^{2k}}(a\pa _b+c\pa_d)^{k}f=(-1)^k(a\pa_c+b\pa_d)^kf.
    \end{equation}
    Since both sides are functions in $\O_\infty^\la$, they are invariant under the action of $\Z_p$ given by $(h\cdot g)(\abcd)=h\cdot (g(\ma{a+hc&b+hd\\c&d}))$. Thus, to prove that (\ref{main}) is true for any $\abcd\in \GL_2(\Z_p)$, we may utilize the action of $\Z_p$ to reduce to the case $a=0$ or $b=0$. 
    
    When $a=0$, the LHS becomes $(-1)^k(b\pa_d)^k$, which equals the RHS. 

    When $b=0$, observe that since $f\in \O_\infty^{\la,(1-k,0)}(U)$, $(a\pa_a+b\pa_b)f=(1-k)f, (c\pa_c+d\pa_d)f=0$ and $(c\pa_a+d\pa_b)f=0$. Since $b=0$, we deduce that $\pa_bf=-\frac{c}{d}\pa_af=\frac{c}{ad}(k-1)f$ and $\pa_df=-\frac{c}{d}\pa_cf$. Note that $(a\pa_b+c\pa_d)^if$ is still in $\O_\infty^{\la,(1-k,0)}(U)$ for any $i\in \N$, thus the LHS of (\ref{main}) when $b=0$ is
    $$\frac{a^kd^k}{c^{2k}}(\frac{c}{d}(k-1)-\frac{c^2}{d}\pa_c)^kf=\frac{a^k}{c^{2k}}(c(k-1)-c^2\pa_c)^kf=(-1)^k\frac{a^k}{c^{2k}}(c^2\pa_c-(k-1)c)^kf.$$
    Thus it suffices to prove 
    $$(c^2\pa_c-(k-1)c)^kf=c^{2k}\pa_c^kf.$$
    Since this identity only depends on the variable $c$, it suffices to prove it for monomials $f=c^i$, $i\in\Z$. A direct computation shows
    $$(c^2\pa_c-(k-1)c)^k(c^i)=\prod_{j=0}^{k-1}(i-(k-1)+j)c^{i+k}=\prod_{j=0}^{k-1}(i-j)c^{i+k},$$
    and
    $$c^{2k}\pa_c^k(c^i)=\prod_{j=0}^{k-1}(i-j)c^{i+k}.$$
    This confirms the equality.
    
\end{pf}

\section{Applications to completed cohomology}\label{app}
Consider the completed cohomology 
\[\tilde{H}^i(K^p,\Q_p):=(\varprojlim_n\varinjlim_{K_p\subset \mathrm{GL}_2(\Q_p)}H^i_{\text{\'et}}(X_{K^pK_p,\overline{\Q}},\Z/p^n\Z))[\frac{1}{p}].\]
This is a $p$-adic Banach space, equipped with natural continuous actions of $\mathrm{GL_2(\Q_p)}$ and $\gal_{\Q}$. Let $\O_{K^p}=\pi_{\mathrm{HT}*}\O_{\mathcal{X}_{K^p}}$ be the push-forward of the structure sheaf along $\pi_{\mathrm{HT}}$ and let $\O_{K^p}^\la\subset \O_{K^p}$ be the subsheaf of $\mathrm{GL}_2(\Q_p)$-locally analytic sections. Then by \cite[Theorem 4.4.6]{panI}, we have
\[\tilde{H}^i(K^p,\Q_p)\hat\otimes_{\Q_p}C\cong H^i(\mathcal{X}_{K^p},\mathcal{O}_{\mathcal{X}_{K^p}})\cong H^i(\fl,\mathcal{O}_{K^p}),\]
and 
\[\tilde{H}^i(K^p,\Q_p)^\la\hat\otimes_{\Q_p}C\cong H^i(\mathcal{X}_{K^p},\mathcal{O}_{\mathcal{X}_{K^p}}^\la)\cong H^i(\fl,\mathcal{O}_{K^p}^\la).\]

As in the previous section, $\mathfrak{g}^0:=\mathfrak{gl}_2(\Q_p)\otimes_{\Q_p}\mathcal{O}_{\fl}$ acts on $\mathcal{O}_{K^p}^\la$. By \cite[Theorem 4.2.7]{panI}, $\mathfrak{n}^0$ acts trivially on $\mathcal{O}_{K^p}^\la$. Thus there is an action of $\mathfrak{b}^0/\mathfrak{n}^0\cong \mathfrak{h}\otimes_{\Q_p}\mathcal{O}_{\fl}$ on $\mathcal{O}_{K^p}^\la$. Let $\theta_\h$ be the horizontal action of $\h$ on $\mathcal{O}_{K^p}^\la$ induced by the natural embedding $\h\hookrightarrow \b^0/\n^0$. For $n_1,n_2\in \Z$, let $\chi=(n_1,n_2)$ be the character of $\h$ sending $\ma{x&0\\0&y}$ to $n_1x+n_2y$. Let $\mathcal{O}_{K^p}^{\la,(n_1,n_2)}$ be the weight-$\chi$ subsheaf of $\mathcal{O}_{K^p}^\la$ with respect to  $\theta_\h$.  Let
$$\Tilde{\chi}_k=\{(0,1-k),(-k,1)\}\subset \h^*$$ 
be the infinitesimal character of the $(k-1)$-th symmetric power of the dual of the standard representation. Then by the relation between $\theta_\h$ and the infinitesimal character \cite[Corollary 4.2.8]{panI}, on $\mathcal{O}_{K^p}^{\la,\Tilde{\chi}_k}$ we have a natural decomposition
$$\mathcal{O}_{K^p}^{\la,\Tilde{\chi}_k}=\mathcal{O}_{K^p}^{\la,(1-k,0)}\oplus \mathcal{O}_{K^p}^{\la,(1,-k)}.$$

In \S 4 of \cite{pan2022}, Pan defined two differential operators 
\begin{align*}
   &d^{k}:\O_{K^p}^{\la,(1-k,0)}\ra \O_{K^p}^{\la,(1-k,0)}\otimes_{\O_{K^p}^\sm}(\Omega^1_{K^p}(\mathcal{C})^\sm)^{\otimes k+1},\\
   &\overline{d}^{k}:\O_{K^p}^{\la,(1-k,0)}\ra \O_{K^p}^{\la,(1-k,0)}\otimes_{\O_{\P^1}}(\Omega^1_{\P^1})^{\otimes k+1}
\end{align*} 
and the intertwining operator 
\begin{align*}
I_{k-1}:=d^{k}\circ \overline{d}^{k}= \overline{d}^{k}\circ d^{k}:\O_{K^p}^{\la,(1-k,0)}&\ra (\Omega^1_{K^p}(\mathcal{C})^\sm)^{\otimes k+1}\otimes_{\O_{K^p}^\sm}\O_{K^p}^{\la,(1-k,0)}\otimes_{\O_{\mathscr{F}\ell}}(\Omega^1_{\P^1})^{\otimes k+1}\\
&\cong \O_{K^p}^{\la,(1,-k)}(k).
\end{align*}
Intuitively, $d^k$ is $\O_{\fl}$-linear and is induced by the theta map
$\theta: \omega^{-k,\sm}\ra \omega^{k+2,\sm}$ on the finite level of modular curves, where $\omega^{i,\sm}:=\varinjlim_{K_p\subset \GL_2(\Z_p)} \omega^i_{K^pK_p}$ and $\omega_{K^pK_p}$ is the pushforward of the sheaf of differentials on the universal elliptic curve. Similarly, $\overline{d}^k$ is $\O_{K^p}^\sm$-linear and is induced by the $k$-th derivation on $\fl$. It follows from the construction that for any open subset $U\subset \P^1(C)$ and any open subset $V\subset \P^1(\Z_p)$ such that $V\subset U$, the following diagram commutes:
\[\begin{tikzcd}
        \O_{K^p}^{\la,(1-k,0)}(U)\arrow[r,"I_{k-1}"]\arrow[d] &\O_{K^p}^{\la,(1,-k)}(U)(k)\arrow[d]\\
        \O_{\infty}^{\la,(1-k,0)}(V)\arrow[r,"I_{k-1}"] &\O_{\infty}^{\la,(1,-k)}(V)(k),
    \end{tikzcd}\]
where the vertical maps are given by restriction.

The intertwining operator $I_{k-1}$ induces a map between the cohomology groups
\[I_{k-1}^1:H^1(\fl,\O_{K^p}^{\la,(1-k,0)})\ra H^1(\fl,\O_{K^p}^{\la,(1,-k)})(k).\]

Consider the de Rham sheaf $\B_{\dr,\X_{K^p}}^+$ on $\X_{K^p}$. Set $\B_{\dr}^+:=\pi_{\HT*}\B_{\dr,\X_{K^p}}^+$. For an integer $i\geq 0$, set $\B_{\dr,i}^+:=\B_{\dr}^+/(t^i)$. We denote by $\B_{\dr,i}^{+,\la}\subset \B_{\dr,i}^+$ the subsheaf of $\mathrm{GL_2(\Q_p)}$-locally analytic sections and by $\B_{\dr,i}^{+,\la,\Tilde{\chi}_k}\subset \B_{\dr,i}^{+,\la}$ the $\Tilde{\chi}_k$-isotypic part. By \cite[6.2.4]{pan2022}, the Fontaine operator on $\B_{\dr,k+1}^{+,\la,\Tilde{\chi}_k}$ defines a map
\[N_k: \O_{K^p}^{\la,(1-k,0)}\ra \O_{K^p}^{\la,(1,-k)}(k). \]
This operator is also compatible with the Fontaine operator at cusps; that is, the following diagram commutes:
\[\begin{tikzcd}
        \O_{K^p}^{\la,(1-k,0)}(U)\arrow[r,"N_{k}"]\arrow[d] &\O_{K^p}^{\la,(1,-k)}(U)(k)\arrow[d]\\
        \O_{\infty}^{\la,(1-k,0)}(V)\arrow[r,"N_{k}"] &\O_{\infty}^{\la,(1,-k)}(V)(k).
    \end{tikzcd}\]

We can also consider the cohomology group $H^1(\fl,\B_{\dr,k+1}^{+,\la,\Tilde{\chi}_k})$. By \cite[Corollary 6.2.15]{pan2022}, the Fontaine operator on it defines a map
\[N_k^1:H^1(\fl,\O_{K^p}^{\la,(1-k,0)})\ra H^1(\fl,\O_{K^p}^{\la,(1,-k)})(k),\]
which coincides with the map induced by $N_k$.

Set $c_k=\frac{1}{k((k-1)!)^2}$. The main theorem of this section is the following.
\begin{thm}\label{n=ddcoh}
    $N_k^1=c_k I_{k-1}^1$ as maps from $H^1(\fl,\mathcal{O}_{K^p}^{\la,(1-k,0)})$ to $H^1(\fl,\mathcal{O}_{K^p}^{\la,(1,-k)})(k)$. 
\end{thm}
\begin{pf}
    Let $U_1=\{x\in \fl\ |\ |x|\leq 1\}$ and  $U_2=\{x\in \fl\ |\ |x|\geq 1\}$ be a cover of $\fl$ with $U_{12}:=U_1\cap U_2=\{x\in \fl\ |\ |x|= 1\}$. Then it follows from \cite[Proposition 4.3.15, Lemma 5.1.2 (1)]{panI} that $H^1(\fl,\mathcal{O}_{K^p}^{\la,(1-k,0)})$ is computed by the cokernel of 
    \[\mathcal{O}_{K^p}^{\la,(1-k,0)}(U_1)\oplus \mathcal{O}_{K^p}^{\la,(1-k,0)}(U_2)\ra \mathcal{O}_{K^p}^{\la,(1-k,0)}(U_{12}).\]
    For any cohomology class $\overline{f}\in H^1(\fl,\mathcal{O}_{K^p}^{\la,(1-k,0)})$, let $f$ be its representative in $\mathcal{O}_{K^p}^{\la,(1-k,0)}(U_{12})$ and set $g=(N_k-c_kI_{k-1})f$. For any $x_0\in U_{12}\cap \P^1(\Z_p)=\Z_p^\times$,  set $\gamma=(\begin{smallmatrix}
        1 & x_0\\
        0 & p
    \end{smallmatrix})$, which induces $x\mapsto px+x_0$ on $\P^1$. Then $U_1\gamma\subset  U_{12}$, thus $\gamma g$ defines a section on $U_1$. By \cite[Lemma III.3.16]{Sch15} there exists $0<\epsilon<\frac{1}{2}$ such that 
    \[\mathcal{X}_{K^p}(\epsilon)_a\subset \pi_{\mathrm{HT}}^{-1}(U_1).\]
    where $\mathcal{X}_{K^p}(\epsilon)_a$ is the $\epsilon$-overconvergent neighborhood of the anticanonical locus at infinite level. Let $\mathcal{C}$ be a collection of cusps of $\mathcal{X}$ such that each connected component of $\mathcal{X}$ contains at least one $x\in \mathcal{C}$. Define $\mathcal{D}_{\mathcal{C},\Gamma(p^\infty)}$ as the pullback
    \[\begin{tikzcd}
        \D_{\mathcal{C},\Gamma(p^\infty)}\arrow[r,hook]\arrow[d] &\mathcal{X}_{K^p}(\epsilon)_a\arrow[d]\\
        \sqcup_{x\in \mathcal{C}} \D\arrow[r,hook]&\X(\epsilon)
    \end{tikzcd}\]
    By \cite[Proposition 6.1]{heuer2022cusps}, the map $\mathcal{O}(\mathcal{X}_{K^p}(\epsilon)_a)\ra \mathcal{O}(\D_{\mathcal{C},\Gamma(p^\infty)})$ is injective. Since  $\D_{\mathcal{C},\Gamma(p^\infty)}$ is a finite disjoint union of copies of $\mathcal{D}_{\Gamma(p^\infty)}$ considered in the previous section, it follows from Theorem \ref{n=ddbar} that $\gamma g=0$ on $\D_{\mathcal{C},\Gamma(p^\infty)}$. Therefore 
    $\gamma g=0$ on $ \mathcal{X}_{K^p}(\epsilon)_a$. By the proof of Lemma III.3.20 in \cite{Sch15}, there exists a rational subset $V$ of $U_1$ containing $0$ such that \[\pi^{-1}_{\mathrm{HT}}(V)\subset \mathcal{X}_{K^p}(\epsilon)_a.\]
    Then $g=0$ on $V\gamma$. Since $g\in \mathcal{O}_{K^p}^{\la}(U_{12})$, it follows from \cite[Theorem 4.3.9]{panI} that we can choose an open compact subgroup $K_p\subset \mathrm{GL}_2(\Z_p)$ and an affinoid subset $U_0\subset \mathcal{X}_{K^pK_p}$ whose preimage in $\mathcal{X}_{K^p}$ equals to $\pi^{-1}_{\HT}(U_{12})$, and locally analytic functions $z_1,z_2,z_3\in \mathcal{O}_{K^p}^{\la}(U_{12})$ such that there exists unique elements $c_{ijk}\in H^0(U_0,\mathcal{O}_{K^pK_p})$ for which
    \[g=\sum_{i,j,k\geq 0}c_{ijk}z_1^iz_2^jz_3^k\]
    holds in $\mathcal{O}_{K^p}^{\la}(U_{12})$. Since the same holds for $g|_{V\gamma}$, from the uniqueness we know that there is an affinoid subset $V_0\subset \mathcal{X}_{K^pK_p}$ whose preimage in $\mathcal{X}_{K^p}$ equals to $\pi^{-1}_{\HT}(V\gamma)$ and $c_{ijk}|_{V_0}=0$ for all $i,j,k\geq 0$. Since $U_0$ is smooth and this vanishing holds for any choice $x_0\in \Z_p^\times$ so that it intersects with every connected component of $U_0$, we deduce that $c_{ijk}=0$ for all $i,j,k\geq 0$. Therefore $g=0$. This implies that $N_k^1=c_kI_{k-1}^1$.
\end{pf}
~\\

As a corollary,  we recover the following result regarding the regular de Rham representations in the locally analytic vectors of completed cohomology, which was first established by Pan in \cite[Theorem 7.2.2]{pan2022}. Let $S$ be a finite set of rational primes containing $p$ such that $K_l\subset \GL_2(\Q_l)$ is maximal for $l\notin S$. 
Let $E$ be a finite extension of $\Q_p$ and let  
\[\rho:\Gal_\Q\ra \GL_2(E)\]
be a two-dimensional continuous absolutely irreducible representation such that $\rho$ is unramified outside $S$ and $\rho|_{\gal_{\Q_p}}$ is de Rham of Hodge-Tate weights $0,k$ for some integer $k>0$. For $l\notin S$, we denote by $T_l$ the double coset action of 
\[[K_l\begin{pmatrix}
    l&0\\0&1
\end{pmatrix} K_l]\]
and by $S_l$ the action of $\begin{pmatrix}
    l&0\\0&l
\end{pmatrix}$. For an open compact subgroup $K_p$ of $\GL_2(\Q_p)$, we define
\[\mathbb{T}(K^pK_p)\subset \mathrm{End}_{\Z_p}({H^1(X_{K^pK_p}(\mathbb{C}),\Z_p)})\]
as the $\Z_p$-subalgebra generated by the Hecke operators $T_l,S_l^{\pm 1}$ for $l\notin S$. Define the Hecke algebra of tame level $K^p$ to be
\[\mathbb{T}(K^p):=\varprojlim_{K_p}\mathbb{T}(K^pK_p).\]
It acts faithfully on $\tilde{H}^1(K^p,\Z_p)$. Moreover, by the Eichler--Shimura relation and \cite{BLR}, there exists a homomorphism $\lambda:\mathbb{T}(K^p)\ra E$ such that the $\rho$-isotypic component of $\tilde{H}^1(K^p,E)$  coincides with the $\lambda$-isotypic component, i.e.
\[\tilde{H}^1(K^p,E)[\lambda]=\tilde{H}^1(K^p,E)[\rho]. \]
\begin{cor}
    There is a natural $\GL_2(\Q_p)$-equivariant isomorphism 
    \[\tilde{H}^1(K^p,C\otimes_{\Q_p}E)[\lambda]^\la_0\cong (\ker I^1_{k-1}\otimes_{\Q_p} E)[\lambda],\]
    where $\tilde{H}^1(K^p,C\otimes_{\Q_p}E)[\lambda]^\la_0$ is the Hodge-Tate weight $0$ part of $\tilde{H}^1(K^p,C\otimes_{\Q_p}E)[\lambda]^\la$.
\end{cor}
\begin{pf}
    It follows from Theorem \ref{n=ddcoh} and the primitive comparison theorem in the proof of \cite[Theorem 7.2.2]{pan2022}. Explicitly, there is a natural isomorphism
    \[\tilde{H}^1(K^p,B_{\dr,k+1}^+\otimes_{\Q_p}E)[\lambda]^\la\cong H^1(\fl, \B_{\dr,k+1}^{+,\la,\tilde{\chi}_k})\otimes_{\Q_p}E[\lambda]. \]
    Thus the Fontaine operator on the left hand side corresponds to 
    \[N_{k}^1[\lambda]:H^1(\fl,\O_{K^p}^{\la,(1-k,0)})[\lambda]\ra H^1(\fl,\O_{K^p}^{\la,(1,-k)})[\lambda]. \]
    Since $\rho$ is de Rham, it follows from Theorem \ref{N=0} that $N_k^1[\lambda]=0$. Therefore, by Theorem \ref{n=ddcoh} we know that 
    \[\tilde{H}^1(K^p,C\otimes_{\Q_p}E)[\lambda]^\la_0\cong \ker N^1_{k}\otimes_{\Q_p} E[\lambda]\cong (\ker I^1_{k-1}\otimes_{\Q_p} E)[\lambda]. \]
\end{pf}

\iffalse
where
\begin{align*}
    \bdr^+(\D_{\infty})=\{\sum_{n\in \Z[\frac{1}{p}]_{\geq 0}} a_nq^n\in B_\dr^+[[q^{1/p^\infty}]]
    \text{ such that }|a_n|q^n\ra 0 \text{ for all } 0\leq q<1 \\\text{and } |a_n|\ra 0 \text{ on bounded intervals} \}
\end{align*}
\fi

\bibliographystyle{amsalpha}
\bibliography{ref.bib}
\end{document}